\documentclass[a4paper]{amsart}

\usepackage{amsmath,amsfonts,amssymb,amsthm}
\usepackage[colorlinks=true]{hyperref}
\usepackage{graphicx}
\usepackage[all]{xy}
\usepackage{mathrsfs}
\usepackage{enumerate}
\usepackage{oldgerm}
\usepackage{MnSymbol}
\usepackage{stmaryrd}

\theoremstyle{definition}
\newtheorem{theorem}{Theorem}[section]

\newtheorem*{theorem*}{Theorem}

\newtheorem{example}[theorem]{Example}
\newtheorem{examps}[theorem]{Examples}
\newtheorem{lemma}[theorem]{Lemma}
\newtheorem{proposition}[theorem]{Proposition}
\newtheorem{remark}[theorem]{Remark}
\numberwithin{equation}{section}

\newcounter{conta}

\newcommand{\toto}{\rightrightarrows}

\def\<{\langle}
\def\>{\rangle}
\newcommand{\action}{\curvearrowright}

\newcommand{\R}{\mathbb R}

\newcommand{\Z}{\mathbb Z}

\newcommand{\Ss}{\mathbb S}

\newcommand{\id}{{\rm id}}

\newcommand{\hol}{{\rm Hol}}

\newcommand{\eps}{\varepsilon}

\renewcommand{\d}{{\rm d}}

\newcommand{\F}{\mathcal F}

\newcounter{comments}

\begin{document}

\title{On deformations of compact foliations}

\author{Matias del Hoyo}
\address{Departamento de Geometria - IME, Universidade Federal Fluminense. Rua Professor Marcos Waldemar de Freitas Reis, 24210-201, Niter\'oi, Brazil}
\email{mldelhoyo@id.uff.br}

\author{Rui Loja Fernandes}
\address{Department of Mathematics, University of Illinois at Urbana-Champaign, 1409 W. Green Street, Urbana, IL 61801, USA} 
\email{ruiloja@illinois.edu}

\thanks{MdH was partially supported by CNPq. RLF was partially supported by NSF grants DMS 1405671, DMS 1710884 and a Simons Fellowship. Both authors acknowledge the support of the \emph{Ci\^encias Sem Fronteiras} grant 401817/2013-0.}

\begin{abstract} 
We combine classic stability results for foliations with recent results on deformations of Lie groupoids and Lie algebroids to provide a cohomological characterization for rigidity of compact Hausdorff foliations on compact manifolds.
\end{abstract}

\maketitle

\vspace{-1cm}

\setcounter{tocdepth}{1} 
\tableofcontents 

\vspace{-1cm}


\section{Introduction}

A foliation $\F$ on a manifold $M$ is {\bf compact Hausdorff} if its leaves are compact and its orbit space is Hausdorff. If $\F$ is compact then its holonomy groups are finite and, by Reeb stability \cite{e}, a small saturated neighborhood of any leaf $L$ is equivalent to its {\bf linearization}. If $M$ is connected then the leaves without holonomy are all diffeomorphic to a {\bf generic leaf} $L_0$, and they comprise a dense open set.

\smallskip

We say that a foliation $\F$ is {\bf rigid} if any 1-parameter deformation $\tilde\F$ of it is obtained from the trivial deformation by conjugating with an isotopy of $M$. In this note we use {\bf Lie groupoids} and {\bf Lie algebroids} \cite{mm} to give a simple proof of the following fundamental result, illustrating the power of this formalism in classic problems of differential geometry and topology. 

\begin{theorem}\label{thm:main}
Let $M$ be a compact connected manifold, and let $\F$ be a compact Hausdorff foliation of $M$. Then $\F$ is rigid if and only if its generic leaf $L_0$ satisfies $H^1(L_0)=0$.
\end{theorem}

A foliation is the same as a Lie algebroid with injective anchor map. Such an algebroid is {\bf integrable} and admits two canonical integrations. One is the {\bf monodromy groupoid} $Mon(\F)\toto M$, which has arrows the homotopy classes of paths within a leaf. The other is the {\bf holonomy groupoid} $Hol(\F)\toto M$, which has arrows germs of transverse diffeomorphisms induced by a path.  $Hol(\F)\toto M$ is the smallest integration and it is a quotient of $Mon(\F)\toto M$, which is the largest source connected integration.

\smallskip


Our proof of Theorem \ref {thm:main} combines classic results such as {\bf Reeb and Thurston stability} \cite{e,t}, with new results on {\bf rigidity of Lie groupoids}, obtained independently in \cite{cms} (which uses cohomological methods) and in \cite{dhf2} (which uses Riemannian metrics on Lie groupoids). We will see later how  the condition in the theorem can be understood in terms of deformation cohomology \cite{cm,cms}, and this leads to a second proof of the theorem using a {\bf Moser's trick} argument for Lie algebroid deformations.

\smallskip

Although Theorem \ref{thm:main}  deals with deformations, i.e., with smooth curves on the {\bf moduli space} of foliations, it is closely related  with the stability results for foliations obtained by Epstein-Rosenberg \cite{er} and Hamilton \cite{ham}. In these works, the authors topologize the space of foliations using the {\bf $C^r$-topology}, and obtain criteria for any nearby foliation to be isomorphic to the given foliation. The precise relationship between those works and our contribution is rather delicate and we leave it to be explored elsewhere.

\vskip 5 pt

\noindent{\bf  Acknowledgments. } We are grateful to I.~Marcut for pointing out a mistake in our first proof of Theorem \ref{thm:main}, and to the anonymous referee for pointing out a mistake in our first version of the cohomological proof. We also thank M.~Crainic, J.N.~Mestre and I.~Struchiner for sharing with us a preliminary version of their preprint \cite{cms}.


%
%
%

\section{Some preliminaries}


Given $M$ a manifold and $\F$ a compact Hausdorff foliation, each leaf $L$ has finite holonomy group $H$, and we can find a small transverse $T$ to $L$ at $x$ such that $H$ acts on $T$ by diffeomorphisms. 
The Reeb Stability Theorem insures the existence of a saturated open $U\supset L$ and a foliated diffeomorphism $\phi:U\to (\tilde L\times T)/H$, where $\tilde L\to L$ is the regular covering corresponding to $H$, $\tilde L\times T$ is foliated by the second projection, and $H$ acts diagonally. Proofs and further details can be found in \cite{e}.


This can be seen as a linearization theorem: the normal bundle to a leaf $\nu(L)\to L$ is foliated by the linear holonomy group $H'$, and if $L'\to L$ is the covering space corresponding to $H'$, then we can recover $\nu(L)\cong (L'\times\R^q)/H'$. Then Reeb stability can be rephrased by saying that, if a leaf $L$ is compact and has finite holonomy $H$, then $H$ agrees with the linear holonomy $H'$, and there is a tubular neighborhood $\phi:\nu(L)\to U\subset M$ defining a foliated diffeomorphism with a saturated open.


The local model provides a neat description of the holonomy groupoid of a compact Hausdorff foliation $\hol(\F_U)\toto U$ restricted to the open $U$, as the quotient of the submersion groupoid $\tilde L\times\tilde L\times T\toto \tilde L\times T$ by the action of $H$. In particular, $\hol(\F_U)$ is Hausdorff and the source map is locally trivial with compact fibers, hence proper. Note that the source-fibers identify with the generic leaf $L_0$. This yields a characterization of compact Hausdorff foliations in terms of their holonomy groupoid. 

\begin{proposition}(cf. \cite[Thm 2.4.2]{cfm})
\label{prop:compact:foliations}
A foliation $\F$ on $M$ is compact if and only if its holonomy groupoid $\hol(\F)\toto M$ is Hausdorff and source-proper.
\end{proposition}

By a {\bf deformation} of a foliation $\F$ parametrized by some interval $0\in I\subset\R$ we mean a foliation $\tilde \F$ on the cylinder $M\times I$ that is tangent to the slices $M\times t$, and that restricts to $\F$ on $M\times 0$. Two deformations are {\bf equivalent} if, after restricting to a smaller interval $J$, they are related by a diffeomorphism fibered over $J$.
A deformation $\tilde \F$ is {\bf trivial} if it is equivalent to the product foliation $\F\times 0_I$. A foliation admiting only trivial deformations is called {\bf rigid}.

\smallskip


Foliations can be seen as Lie algebroids with injective anchor map. Given $A$ a Lie algebroid over $M$, a {\bf Lie algebroid deformation} $\tilde A$ is a Lie algebroid structure on the vector bundle $A\times 0_I
$ over the cylinder $M\times I$ such that the image of the anchor map $\rho(\tilde A)$ is included in $TM\times 0_I$, and such that the central fiber $\tilde A|_{M\times 0}$ is the original algebroid. Equivalent deformations and rigidity are defined as before.
Since the injectivity of the anchor map is an open condition we get the following:

\begin{lemma} (cf. \cite{cm})
If $M$ is compact then any Lie algebroid deformation $\tilde A$ of a foliation $\F$ is equivalent to a foliation deformation.
\end{lemma}


Similarly, given $G\toto M$ a Lie groupoid, a {\bf Lie groupoid deformation} is a Lie groupoid structure $G\times I\toto M\times I$ over the cylinder whose orbits are included in the slices $M\times t$ and such that it restricts to the original groupoid at time 0. Note that we are deforming the structure maps source, target, multiplication, unit and inverse, but keeping the manifolds of objects and arrows constant. Equivalences and rigidity are defined as before. Recently, the rigidity of compact Lie groupoids has been established independently in \cite{cms}, using a deformation cohomology theory for Lie groupoids, and in \cite{dhf2}, using the theory of Riemannian Lie groupoids:

\begin{theorem}[{\cite[Thm 7.4]{cms},\cite[Thm 5.0.3]{dhf2}}]
\label{thm:rigidity}
A compact, Hausdorff, Lie groupoid $G\toto M$ is rigid.
\end{theorem}




Starting with a foliation $\F$, any groupoid deformation of $\hol(\F)$ yields a deformation of $\F$ by differentiation, but the inverse procedure of integrating deformations is more subtle. Even though a deformation, viewed as a foliation $\F$ over the cylinder, can be integrated to its holonomy or monodromy groupoid, its integration may not be a deformation as defined above. For instance, the arrow manifold may differ from the cylinder $\hol(\F)\times I$, as shown in the following example.

\begin{example}(c.f. \cite{lr})
\label{example:foliation:H1}
Let $L$ be a compact manifold with $H^1(L)\neq 0$ and let $\F$ be the foliation on $L\times \Ss^1$ given by the second projection.
The forms $\eps\omega+\d \theta$ and $\d\eps$, where $[\omega]\neq 0 \in H^1(L)$, define a foliation deformation $\tilde \F$ on $L\times\Ss^1\times I$ tangent to the fibers of the projection 
$L\times\Ss^1\times I\to I$, $(x,\theta,\eps)\mapsto\eps$.
This deformation is non-trivial: if $\eps\neq 0$ then the leaves of $\F_\eps$ are non-trivial coverings of $L$. 
When integrating the deformation to the holonomy groupoid, $\hol(\tilde \F)\toto L\times\Ss^1\times I$, this is not a groupoid deformation, for the manifold of arrows is not constant in time, namely $\hol(\tilde \F)\neq \hol(\F)\times I$. In fact, note that at $\eps=0$ the holonomy groupoid $\hol(\F)=L\times L\times\Ss^1\toto L\times\Ss^1$ is compact. However, for $\eps\not=0$ the groupoid $\hol(\tilde \F_\eps)$ does not have compact source fibers and hence is not compact. 
\end{example}

\section{Proof of main Theorem}




Given $M$ a manifold, $\F$ a foliation, and $\tilde \F$ a foliation deformation, we can identify a leaf $L$ of $\F$ with the leaf $L\times 0$ of $\tilde \F$, and compare both holonomies, by restricting to $M\times 0$ a local transverse $\tilde T$ to $L\times 0$ at $(x,0)$ within $M\times I$.
$$r:\hol_{L\times 0}(\tilde \F)\to \hol_L(\F)$$
This map is clearly onto, and it is an isomorphism for a trivial deformation, but it might have a non-trivial kernel $K$, as in Example \ref{example:foliation:H1}, where $K=\Z$. Thurston stability assert that if a compact leaf satisfies $H^1(L)=0$ then either it has trivial holonomy or it has non-trivial linear holonomy. Next we use a variant of it to show that $K=0$ for compact Hausdorff foliations.

\begin{proposition}\label{prop:hol=hol}
If $M$ is compact, $\F$ is compact and $H^1(L_0)=0$, then $r$ gives a group isomorphism $\hol_{L\times 0}(\tilde \F)= \hol_L(\F)$ for any leaf $L$ of $\F$.
\end{proposition}

\begin{proof}
Let us first show that the restriction map $r:\hol_{L\times 0}(\tilde \F)\to \hol_L(\F)$ induces an isomorphism on the linear holonomies $\d r:\d\hol_{L\times 0}(\tilde\F)\to \d\hol_L(\F)$, or equivalently, that the kernel $K'$ of $\d r$ is trivial.
Using coordinates $(x,t)$ with $x$ in $M$ and $t\in I$, we can represent the linear holonomy of a loop $\gamma$ as a matrix as below.
$$\d[\gamma]=\begin{bmatrix}
   \d r[\gamma] & v[\gamma]\\ 0 & 1
  \end{bmatrix}
$$

It follows from the local model and Reeb stability that the fundamental group of the generic leaf $\pi_1(L_0)$ is the kernel of the projection $\pi_1(L)\to \hol_L(\F)=\d \hol_L(\F)$. Therefore, there is an epimorphism $\pi_1(L_0)\to K'$. If $K'$ were not trivial, then the formula $\gamma\mapsto v[\gamma]$ would define a non-trivial group homomorphism $\pi_1(L_0)\to\R^n$, which is an absurd, since $H^1(L_0)=0$.

Suppose now that $L\cong L_0$ is a generic leaf, or equivalently, that $\hol_L(\F)=0$. Then, the linear holonomy of $\tilde\F$ at $L\times 0$ is also trivial, and we can conclude that $\hol_{L\times 0}(\tilde\F)=0$ by Thurston stability \cite[Thm2]{t}.

Finally, when $L$ is any leaf, we can reduce to the previous case by first restricting our attention to a small tubular neighborhood $U$ of $L$, and then considering the covering space $p:\tilde U\to U$ corresponding to $\pi_1(L_0)\subset\pi_1(L)$. The pullback foliations $p^*\F$ and $p\times\id_I^*\tilde\F$ have trivial holonomy at $\tilde L=p^{-1}(L)$ and $\tilde L\times 0$ respectively, so we can apply Thurston stability as in the previous case.
\end{proof}


We can use the previous proposition to integrate a foliation deformation to a groupoid deformation, if we use the holonomy groupoid.

\begin{proposition}\label{prop:H1}
Let $\F$ be a compact Hausdorff foliation on a compact connected manifold $M$ whose generic leaf $L_0$ satisfies $H^1(L_0)=0$. If $\tilde \F$ is a deformation of $\F$ then the restriction $\tilde \F|_{M\times J}$ to a smaller interval $J\subset I$ is a compact Hausdorff foliation, and its holonomy groupoid $\hol(\tilde \F|_{M\times J})$ is a groupoid deformation.
\end{proposition}

\begin{proof}
By Proposition \ref{prop:hol=hol} we have that every leaf of $\tilde\F$ at time zero is compact with finite holonomy, hence it admits a local linear model. 
By Reeb stability, the orbit space $(M\times I)/\tilde\F$ is Hausdorff around $L\times 0$, and by compactness of $M$, it is so after restricting the cylinder to $M\times J$, with $J\subset I$ a smaller interval. 
Hence $\tilde\F|_{M\times J}$ is a compact Hausdorff foliation and its holonomy groupoid $\hol(\tilde\F)\toto M\times J$ is Hausdorff and source-proper.
The identification $\hol(\tilde\F)\cong\hol(\F)\times J$ fibered over $J$ follows from a semi-local version of Ehresmann theorem, namely the linearization of the source map around $M\times 0\subset M\times J$ \cite[Cor. 5.14]{dhf1}.
\end{proof}


We can now use rigidity of compact Lie groupoids to give a simple proof of rigidity of compact Hausdorff foliations:

\begin{proof}[Proof of Theorem \ref{thm:main}]
If the generic leaf $L_0$ satisfies $H^1(L_0)= 0$, and $\tilde \F$ is a deformation of $\F$ then, by Proposition \ref{prop:H1}, the groupoid $\hol(\tilde\F)$ can be regarded as a proper groupoid deformation of the compact groupoid $\hol(\F)\toto M$. Hence, by Theorem \ref{thm:rigidity}, $\hol(\tilde\F)$ is locally trivial and so is the deformation $\tilde \F$.

If the generic leaf $L_0$ satisfies $H^1(L_0)\ne 0$, we can adapt Example \ref{example:foliation:H1} to construct a non-trivial deformation. Let $B$ be a small ball with coordinates $t_1,\dots,t_k$ and define $\tilde\F$ on $L\times B\times I$ by the forms $\eps\lambda(t)\omega+\d t_1,\d t_2,\dots,\d t_k$, where $\lambda$ is such that $\lambda(0)=1$ and $\lambda(t)=0$ for $t$ close to $\partial B$. This is a non-trivial deformation of the product foliation on $L\times B$ that remains constant on the border. We can copy this deformation in a foliated tubular neighborhood $L\subset T\subset M$ of a generic leaf of $\F$, and extend it outside $T$ stationarily.
\end{proof}

\section{The cohomological proof}


A cohomological approach to deformations of Lie groupoids and Lie algebroids has been developed in \cite{cm,cms}. Every Lie algebroid $A$ has a {\bf deformation complex} $C_{\textrm{def}}(A)$ which can be defined as the cohomology of $A$ with coefficients on the adjoint representation (in general, a representation up to homotopy  \cite{ac}). Every Lie algebroid deformation $\{A_\eps:\eps\in I \}$ of $A$ yields a cocycle 
$$c_0=\left.\frac{\d}{\d \eps}\right|_{\eps=0}[,]_\eps\in C^2_{\rm def}(A)$$
and its class $[c_0]\in H^2_{\textrm{def}}(A)$ is invariant by equivalences of deformations \cite{cm}.


When $A=\F$ is a foliation on $M$, the adjoint representation is quasi-isomorphic to the representation $\F\action \nu(\F)$ on the normal bundle given by the {\bf Bott connection}, and the deformation cohomology of $\F$ agrees with the shifted Lie algebroid cohomology with coefficients, namely $H^\bullet_{\textrm{def}}(\F)\cong H^{\bullet-1}(\F,\nu(\F))$. This way, given a deformation $\tilde\F=\{\F_\eps:\eps\in I\}$ the class 
$[c_0]\in H^2_{\textrm{def}}(\F)$ 
constructed by Crainic and Moerdijk corresponds to the deformation cohomology class investigated by Heistch \cite{h}. Namely, the class of the $\F$-foliated cocycle with values in $\nu$ given by:
\begin{equation}
\label{eq:cocycle:Heistch}
c_0(v):=\pi^\perp_0\left(\left.\frac{\d}{\d \eps}\right|_{\eps=0}\pi_\eps(v)\right),\quad v\in \F
\end{equation}
where $\pi_\eps:TM\to \F_\eps$ and $\pi_\eps^\perp:TM\to \nu(\F_\eps)$ are orthogonal projections relative to some Riemannian metric.


Crainic and Moerdijk also established a cohomological characterization for trivial deformations when the base manifold $M$ is compact. In the case of foliations, a deformation $\tilde\F=\{\F_\eps:\eps\in I\}$ of a foliation $\F=\F_0$ is trivial if and only if the classes 
$$c_t(v):=\pi^\perp_t\left(\left.\frac{\d}{\d \eps}\right|_{\eps=t}\pi_\eps(v)\right),\quad v\in \F_t$$
vanish {\em smoothly with respect to $t$} (cf. \cite[Thm 2]{cm}). 

\begin{remark}\label{rmk:smooth-vanishing}
Let us discuss in more detail the smooth vanishing of the classes $c_t$. 
The normal bundle to the total foliation $\tilde\F$ has a subbundle given by 
$$K=\ker(\pi:\nu(\tilde\F)\to TI)=(TM\times 0_I)/\tilde \F.$$
The action $\tilde\F\action\nu(\tilde\F)$ given by the Bott connection $\nabla$ preserves $K$, and this leads to a complex $(C^\bullet(\tilde\F,K),\delta)$. The classes $c_t$ define together a global class $c\in C^1(\tilde\F,K)$ determined by $i_t^*c=c_t$ for all $t$, where $i_t:M\times \{t\}\to M\times I$ is the inclusion. The classes $c_t$ to vanish smoothly with respect to $t$ means that $[c]\in H^1(\tilde\F,K)$ is trivial, or equivalently, that there is a vector field $X$ on $M\times I$ tangent to the slices $M\times \{t\}$ such that $\delta(X\ {\rm mod}\ \tilde\F)=c$. The flow of such a time-dependent vector field $X=\{X_t\}$ gives an isotopy trivializing $\tilde\F$ (cf. \cite[Thm 2]{cm}).
\end{remark}


Given $G\toto M$ a Lie groupoid, $A_G$ its Lie algebroid, and $E\to M$ a vector bundle endowed with a representation of $G$, there is an induced representation $A_G\action E$ by differentiation, and the corresponding Lie groupoid and Lie algebroid cohomology are related by the so-called Van Est map
$$\textrm{VE}:H^\bullet(G,E)\to H^\bullet(A_G,E).$$
If the source-fibers have trivial first cohomology then it follows from a standard spectral sequence argument that $\textrm{VE}$ is an isomorphism on degree 1  \cite[Thm 4]{c}. 


\begin{proposition}\label{prop:def-coh}
Let $\F$ be a compact Hausdorff foliation on a compact connected manifold $M$ whose generic leaf $L_0$ satisfies $H^1(L_0)=0$. Then $H^2_{\textrm{def}}(\F)=0$, and more generally, $H^1(\F,E)=0$ for any representation $\F\action E$. 
\end{proposition}

\begin{proof}
Consider the Van Est map corresponding to $G=\hol(\F)$ acting over some vector bundle $E$.
Since the source-fibers of the holonomy groupoid identify with the generic leaf $L_0$, it follows from \cite[Thm 4]{c} that $H^1(\hol(\F),E)\cong H^1(\F,E)$. By Proposition \ref{prop:compact:foliations}, $\hol(\F)$ is a proper groupoid so its positive cohomology groups vanishes for any coefficients \cite[Prop 1]{c}. This proves the second statement. The first one follows from the isomorphism $H^2_{\textrm{def}}(\F)\cong H^{1}(\F,\nu(\F))$ (\cite[Prop 4]{cm}).
\end{proof}

Note that the first statement of the previous proposition can also be proven by directly comparing the deformation cohomology of the foliation and its holonomy groupoid, as in \cite[Thms 6.1 and 10.1]{cms}. Although the statements there demand the source-fibers to be simply connected, the vanishing of the cohomology is enough.

The fundamental fact behind the main theorem is that the holonomy groupoid of a deformation $\tilde \F$  of a compact Hausdorff foliation is a source proper groupoid over some restricted cylinder $M\times J$, and therefore the slices $\tilde\F_t$ are also compact Hausdorff. Once this is established, as in Proposition \ref{prop:H1}, we can give a cohomological version of the proof of the main theorem.

\begin{proof}[Cohomological proof of Theorem \ref{thm:main}]
Let $\tilde\F$ be a deformation of a compact Hausdorff foliation $\F$ with $H^1(L_0)=0$. By Proposition \ref{prop:H1} the restriction $\tilde \F|_{M\times J}$ is also compact, Hausdorff, for some $J\subset I$, and hence so is $\F_t$ for each $t\in J$, By Proposition \ref{prop:def-coh}, the deformation cohomology of $\F_t$ vanishes for $t\in J$. It follows that the deformation cohomology classes $[c_t]\in H^2_{\textrm{def}}(\F_t)$  all vanish for $t\in J$.  We claim that, moreover, these classes vanish smoothly with respect to $t\in J$. As discussed in Remark \ref{rmk:smooth-vanishing}, this means that $[c]\in H^1(\tilde F|_{M\times J},K|_{M\times J})$ is trivial. But again by Proposition \ref{prop:def-coh}, the whole group $H^1(\tilde F|_{M\times J},K|_{M\times J})$ is trivial.
The result now follows by an algebroid version of Moser's trick: a primitive of $c$ is a time-dependent vector field $X=\{X_t\}$ which gives an isotopy trivializing $\tilde\F$ (see \cite[Thm 2]{cm}).
\end{proof}

{

}



\begin{thebibliography}{xxx}

\bibitem{ac}
C Arias-Abad, M Crainic: 
Representations up to homotopy of Lie algebroids.  
{\em J. Reine Angew. Math.} {\bf 663} (2012), 91--126.








\bibitem{c}
M Crainic; 
Differentiable and algebroid cohomology, Van Est isomorphisms, and characteristic classes;
{\em Comment. Math. Helv.} {\bf 78} (2003), 681--721.

\bibitem{cfm}
M Crainic, RL Fernandes, D Martinez Torres; Regular Poisson Manifolds of Compact Types (PMCT 2). 
Preprint  \texttt{arXiv:1603.00064}.



\bibitem{cm}
M Crainic, I Moerdijk;
Deformations of Lie brackets: cohomological aspects;
{\em J. Eur. Math. Soc.} {\bf 10} (2008), 1037--1059. 

\bibitem{cms}
M Crainic, JN Mestre, I Struchiner;
Deformations of Lie Groupoids.
Preprint \texttt{arXiv: 1510.02530}.

%
\bibitem{dhf1}
M del Hoyo, RL Fernandes;
Riemannian Metrics on Lie Groupoids;
{\em J. Reine Angew. Math.} {\bf 735} (2018), 143--173.

\bibitem{dhf2}
M del Hoyo, RL Fernandes;
Riemannian metrics on differentiable stacks; 
to appear in {\em Mathematische Zeitschrift}. Preprint \texttt{arXiv:1601.05616}.

\bibitem{e}
D Epstein;
Foliations with all leaves compact;
{\em Ann. Inst. Fourier} {\bf 26} (1976), 265--282.

\bibitem{er}
D Epstein, H Rosenberg;
Stability of compact foliations. 
In \emph{Geometry and Topology}, Lecture Notes in Mathematics 597, Springer-Verlag (1977), 151--160.

\bibitem{ham}
R. Hamilton, Deformation theory of foliations, unpublished.

\bibitem{h}
J Heitsch; 
A cohomology for foliated manifolds; 
{\em Comment. Math. Helv.} {\bf 50} (1975), 197--218.



%
%
%
\bibitem{lr}
R. Langevin, H. Rosenberg;
Integrable perturbations of fibrations and a theorem of Seifert. In 
\emph{Differential topology, foliations and Gelfand-Fuks cohomology} (Proc. Sympos., Pontifícia Univ. Católica, Rio de Janeiro, 1976), LNM, 652, Springer, Berlin, 1978, p. 122-127.



%

\bibitem{mm}
I. Moerdijk, J. Mrcun;
\emph{Introduction to Foliations and Lie Groupoids}.
Cambridge Studies in Advanced Mathematics 91,
Cambridge University Press (2003).








\bibitem{t}
Thurston,  W.;  
A  generalization  of  the  Reeb  stability  theorem; 
\emph{Topology} {\bf13} (1974), 347--352.


%

\end{thebibliography}
\end{document}